\documentclass{amsart}

\RequirePackage{color}

\let\oldtocsection=\tocsection
 
\let\oldtocsubsection=\tocsubsection
 
\let\oldtocsubsubsection=\tocsubsubsection
 
\renewcommand{\tocsection}[2]{\hspace{0em}\oldtocsection{#1}{#2}}
\renewcommand{\tocsubsection}[2]{\hspace{2em}\oldtocsubsection{#1}{#2}}
\renewcommand{\tocsubsubsection}[2]{\hspace{2em}\oldtocsubsubsection{#1}{#2}}

\def\smallskip{\vskip\smallskipamount}
\def\medskip{\vskip\medskipamount}
\def\bigskip{\vskip\bigskipamount}

\usepackage{amsmath,amsthm,bm,mathrsfs,amscd,tikz}
\usepackage{ytableau}
\usepackage{comment}
\usepackage{mathdots}
\usepackage[all]{xy}
\usepackage{graphicx}
\usepackage[colorinlistoftodos]{todonotes}
\usepackage{enumerate}
\usepackage{feynmp-auto}
\usepackage[english]{babel}
\usepackage[utf8]{inputenc}
\usepackage{float}
\usepackage{tikz}
\usepackage{tikz-cd}
\usepackage{amssymb}
\usepackage{mathtools}
\usepackage{hyperref}
\usetikzlibrary{decorations.pathmorphing}

\usepackage[utf8]{inputenc}

\newtheoremstyle{thmstyle}{}{}{\itshape}{}{\bfseries}{ }{5pt}{}
\newtheoremstyle{exstyle}{}{}{}{}{\bfseries}{ }{5pt}{}
\newtheoremstyle{defstyle}{}{}{}{}{\bfseries}{ }{5pt}{}
\newtheoremstyle{remstyle}{}{}{}{}{\bfseries}{ }{5pt}{}

\theoremstyle{thmstyle}
\newtheorem{thm}{Theorem}[section]
\newtheorem{theorem}[thm]{Theorem}

\newtheorem{lemma}[thm]{Lemma}

\newtheorem{proposition}[thm]{Proposition}

\newtheorem{corollary}[thm]{Corollary}

\newtheorem{conjecture}[thm]{Conjecture}

\theoremstyle{exstyle}

\theoremstyle{defstyle}

\newtheorem{def-prop}[thm]{Definition-Proposition}
\newtheorem{def-lem}[thm]{Definition-Lemma}
\newtheorem{rem-convention}[thm]{Remark-Convention}
\newtheorem{def-note}[thm]{Definition-Notation}

\theoremstyle{remstyle}

\newtheorem{remark}[thm]{Remark}

\newcommand{\Hom}{\operatorname{Hom}}

\newcommand{\Ext}{\operatorname{Ext}}

\DeclareMathOperator*{\modu}{mod}

\DeclareMathOperator*{\rep}{rep}

\DeclareMathOperator*{\End}{End}

\DeclareMathOperator*{\node}{node}

\DeclareMathOperator*{\Mri}{Mri}

\DeclareMathOperator*{\GL}{GL}
\DeclareMathOperator*{\SL}{SL}
\DeclareMathOperator*{\SI}{SI}

\DeclareMathOperator*{\sB}{sB}
\DeclareMathOperator*{\gC}{gC}
\DeclareMathOperator*{\B}{B}

\DeclareMathOperator*{\nD}{nD}

\newtheorem{innercustomthm}{Theorem}
\newenvironment{customthm}[1]
  {\renewcommand\theinnercustomthm{#1}\innercustomthm}
  {\endinnercustomthm}        

\makeatletter
\newcommand{\doublewidetilde}[1]{{%
  \mathpalette\double@widetilde{#1}%
}}
\newcommand{\double@widetilde}[2]{%
  \sbox\z@{$\m@th#1\widetilde{#2}$}%
  \ht\z@=.9\ht\z@
  \widetilde{\box\z@}%
}
\makeatother

\begin{document}

\title[$\tau$-tilting theory and module varieties]{$\tau$-tilting finiteness of non-distributive algebras and their module varieties}
\author[Kaveh Mousavand]{Kaveh Mousavand} 
\address{LaCIM, UQAM, Montréal, Québec, Canada}
\email{mousavand.kaveh@gmail.com }
\thanks{The author is partially supported by ISM Scholarship.}

\subjclass[2010]{16G20,16G60,05E10,14K10}

\maketitle

\begin{abstract}
We treat the $\tau$-tilting finiteness of those minimal representation-infinite (min-rep-infinite) algebras which are non-distributive.
Building upon the new results of Bongartz, we fully determine which algebras in this family are $\tau$-tilting finite and which ones are not. This complements our previous work in which we carried out a similar analysis for the min-rep-infinite biserial algebras.
Consequently, we obtain nontrivial explicit sufficient conditions for $\tau$-tilting infiniteness of a large family of algebras. This also produces concrete families of ``minimal $\tau$-tilting infinite algebras"-- the modern counterpart of min-rep-infinite algebras, independently introduced by the author and Wang.

We further use our results on the family of non-distributive algebras to establish a conjectural connection between the $\tau$-tilting theory and two geometric notions in the study of module varieties introduced by Chindris, Kinser and Weyman. We verify the conjectures for the algebras studied in this note:
For the min-rep-infinite algebras which are non-distributive or biserial, we show that if $\Lambda$ has the dense orbit property, then it must be $\tau$-tilting finite. Moreover, we prove that such an algebra is Schur-representation-finite if and only if it is $\tau$-tilting finite. The latter result gives a categorical interpretation of Schur-representation-finiteness over this family of min-rep-infinite algebras.
\end{abstract}

\tableofcontents

\section{Introduction}\label{Introduction}

\subsection{Notations and setting}
Throughout the paper, $k$ is assumed to be an algebraically closed field. 
By $\Lambda$ we denote a basic, connected, associative algebra with multiplicative identity and $\Lambda$ is finite dimensional as a vector space over $k$. By an ideal, we always mean a two-sided ideal.
Every such $\Lambda$ is determined by a bound quiver $(Q,I)$, where $Q$ is a quiver and $I$ is an admissible ideal in the path algebra $kQ$, such that $\Lambda\simeq kQ/I$.

For the standard materials in representation theory of finite dimensional algebras, see \cite{ASS}. Due to the close connections between the scope of this paper and our recent work \cite{Mo}, we frequently refer to our methodology and results in the aforementioned paper, so the reader can view our new results in a larger framework.

\subsection{Objectives and results }
In \cite{Mo} we began a systematic study of minimal representation-infinite algebras from the viewpoint of the $\tau$-tilting theory introduced in \cite{AIR}.
Recall that $\Lambda$ is \emph{minimal representation-infinite} (or min-rep-infinite, for short) provided that it is rep-infinite but for every nonzero ideal $J$ of $\Lambda$, the quotient algebra $\Lambda/J$ is rep-finite.
As shown in \cite{Ri}, every min-rep-infinite algebra falls into (at least) one of the following three subfamilies: 
\begin{itemize}

\item $\Mri({\mathfrak{F}_{\sB}})$: min-rep-infinite \textbf{s}pecial \textbf{B}iserial algebras;

\item $\Mri({\mathfrak{F}_{\nD}})$: min-rep-infinite \textbf{n}on-\textbf{D}istributive algebras;

\item $\Mri({\mathfrak{F}_{\gC}})$: min-rep-infinite algebras with a \textbf{g}ood \textbf{C}overing $\widetilde{\Lambda}$ such that a finite convex subcategory of $\widetilde{\Lambda}$ is tame-concealed of type $\widetilde{\mathbb{D}}_n$ or $\widetilde{\mathbb{E}}_{6,7,8}$.
\end{itemize}

Building upon the work of Ringel \cite{Ri} and the ``brick-$\tau$-rigid correspondence" of Demonet-Iyama-Jasso \cite{DIJ}, in \cite{Mo} we fully determined which algebras in $\Mri({\mathfrak{F}_{\sB}})$ are $\tau$-tilting finite and which ones are not. 
Here we use the recent results of Bongartz in \cite{Bo2} and do the same for the family $\Mri({\mathfrak{F}_{\nD}})$. Although our methodology is similar to that employed in \cite{Mo}, we need to use some new tools to treat non-distributive algebras.
This is because for an algebra $\Lambda$ in $\Mri({\mathfrak{F}_{\nD}})$, unlike those in $\Mri({\mathfrak{F}_{\sB}})$, there is no concrete classification of all indecomposable $\Lambda$-modules in terms of string and bands.

Based on our study of $\tau$-tilting finiteness of min-rep-infinite algebras, in the second part of the paper we relate certain geometric notions defined in \cite{CKW} in terms of $\rep(\Lambda,\underline{d})$ to the
$\tau$-tilting finiteness of $\Lambda$ for $\Lambda$ in  $\Mri({\mathfrak{F}_{\nD}})$ or $\Mri({\mathfrak{F}_{\sB}})$.

For more details on our methodology and the long-term objectives of this project, we refer the reader to the introduction of \cite{Mo}. All the materials and motivations from the $\tau$-tilting theory that we employ here come from \cite{AIR} and \cite{DIJ}. Furthermore, in \cite{Ba}, \cite{Bo2} and \cite{Ri} the reader can find more explanation on the decisive role of min-rep-infinite algebras in the development of representation theory of associative algebras.

For an algebra $\Lambda$, a module $X$ in $\modu \Lambda$ is called a \emph{brick} if $\End_{\Lambda}(X) \simeq k$. In the more geometric settings, a brick in $\modu \Lambda$ is often called a \emph{Schur} representation of $\Lambda$.
In \cite{DIJ}, the authors find several equivalences for $\tau$-tilting finiteness of an algebra, among which they show $\Lambda$ is $\tau$-tilting finite if and only if it is \emph{brick-finite} (i.e, there are only finitely many isomorphism classes of bricks in $\modu \Lambda$). This equivalence plays a crucial role in our study of the $\tau$-tilting finiteness.

An algebra $\Lambda$ is called \emph{distributive} provided that the ideal lattice is distributive. It is well-known that every rep-finite algebra is distributive.
Suppose $\mathfrak{F}_{\nD}$ denotes the family of all non-distributive algebras. Hence $\mathfrak{F}_{\nD}$ only contains rep-infinite algebras.
As in the list above, $\Mri(\mathfrak{F}_{\nD})$ is the family of those min-rep-infinite algebras which are non-distributive.

In \cite{Ja}, Jans studied distributive algebras and showed that every $\Lambda$ in $\mathfrak{F}_{\nD}$ is \emph{strongly unbounded}, meaning that there exists an infinite sequence of positive integers $d_1 < d_2 < d_3 < \cdots$ such that for each $d_i$ there are infinitely many (isomorphism classes of) indecomposable $\Lambda$-modules of length $d_i$.
His work was primarily motivated by the celebrated Brauer-Thrall conjectures and it was followed by the intensive study of min-rep-infinite algebras to verify the conjectures for the general case. For the statements of the conjectures, we refer to \cite[IV.5]{ASS}. More on the historical aspects of the research in this direction could be found in \cite{Bo2}.

Before we state our results, let us recall some definitions. For $\Lambda=kQ/I$, a vertex $x \in Q_0$ is called a \emph{node} if it is neither a sink nor a source and for every arrow $\alpha$ incoming to $x$ and every arrow $\beta$ outgoing from $x$, we have $\beta\alpha \in I$. If $\node(\Lambda)$ denotes the set of nodes in $\Lambda$, we say $\Lambda$ is \emph{node-free} if $\node(\Lambda)= \emptyset$.
Following \cite{Mo} and \cite{Wa}, we say $\Lambda$ is \emph{minimal $\tau$-tilting infinite} (or min-$\tau$-infinite, for short) if $\Lambda$ is $\tau$-tilting infinite, but $\Lambda/J$ is $\tau$-tilting finite, for each nonzero ideal $J$ in $\Lambda$.

Now we summarize the results of the paper. Let us begin with a general result on the minimal $\tau$-tilting infinite algebras, which we prove in Section \ref{Section:Minimal tau-tilting infinite algebras}. We remark that, unlike the rest of the paper, for the following theorem we do not need to assume that the field $k$ is algebraically closed.

\begin{theorem}
If $k$ is an arbitrary field and $\Lambda=kQ/I$ is minimal $\tau$-tilting infinite,
then $\Lambda$ is node-free.
\end{theorem}

We use this observation and our results in \cite{Mo}, and based on the elegant classification of min-rep-infinite non-distributive algebras by Bongartz \cite{Bo2}, in Section \ref{Section:tau-tilting infinite minimal representation-infinite algebras} we obtain the following theorem. For a quick review of non-distributive algebras and Bongartz's classification, see Section \ref{Priliminary}.

\begin{theorem}\label{Introduction: bound quivers of tau-inf}
Let $\Lambda=kQ/I$ be a minimal representation-infinite algebra which is non-distributive. Then, $\Lambda$ is $\tau$-tilting infinite if and only if $(Q,I)$ belongs to the following list of bound quivers:
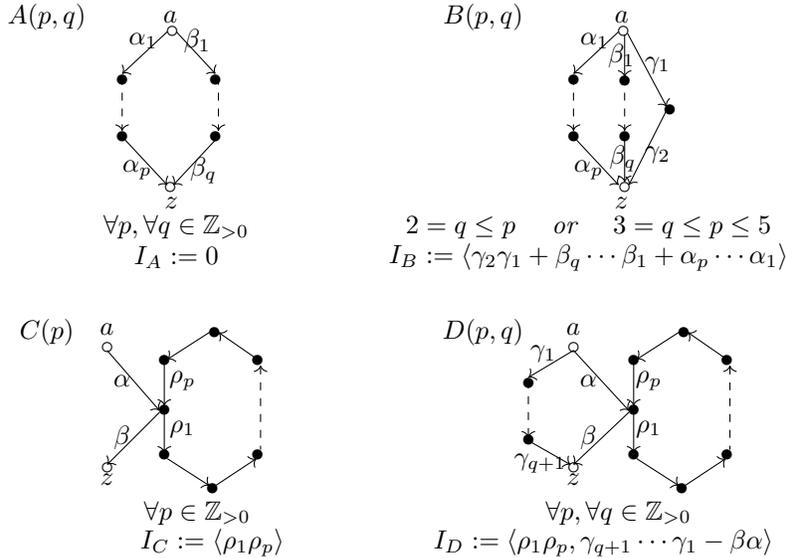
\begin{center}
\begin{tikzpicture}[scale=0.8]

\node at (-1.5,1) {$A(p,q)$};

\node at (0.58,0.75) {$\circ$};
\node at (0.55,1) {$a$};
\draw [->] (0.5,0.75)-- (-0.25,0.05);
\node at (0.1,0.6) {$\alpha_1$};
\node at (-0.25,-0.05) {$\bullet$};

\draw [dashed,->] (-0.25,0)-- (-0.25,-0.9);

\node at (-0.25,-1) {$\bullet$};
\draw [->] (-0.2,-1.05)-- (0.5,-1.8);
\node at (0,-1.55) {$\alpha_p$};

\draw [->] (0.65,0.75)-- (1.3,0.05);
\node at (1,0.6) {$\beta_1$};
    \node at (1.3,-0.05) {$\bullet$};  

\draw [dashed,->] (1.35,0)-- (1.35,-0.9);

\node at (1.3,-1) {$\bullet$};
\draw [->] (1.3,-1.05)--(0.6,-1.8);
\node at (1.1,-1.55) {$\beta_q$};

\node at (0.55,-1.85) {$\circ$};
\node at (0.57,-2.1) {$z$};
\node at (0.65,-2.5) {$\forall p, \forall q \in \mathbb{Z}_{>0}$};
\node at (0.65,-3) {$I_A:=0$};


\node at (5.75,1) {$B(p,q)$};

\node at (8.08,0.75) {$\circ$};
\node at (8.05,1) {$a$};
\draw [->] (8,0.75)-- (7.25,0.05);
\node at (7.6,0.6) {$\alpha_1$};
\node at (7.25,-0.05) {$\bullet$};

\draw [dashed,->] (7.25,0)-- (7.25,-0.9);

\node at (7.25,-1) {$\bullet$};
\draw [->] (7.3,-1.05)-- (8,-1.8);
\node at (7.5,-1.55) {$\alpha_p$};

\draw [->] (8.1,0.7)-- (8.1,0);
\node at (8.05,0.35) {$\beta_1$};
    \node at (8.1,-0.07) {$\bullet$};  

\draw [dashed,->] (8.1,0)-- (8.1,-0.9);

\node at (8.1,-1) {$\bullet$};
\draw [->] (8.1,-1.05)--(8.1,-1.75);
\node at (8.07,-1.35) {$\beta_q$};

\node at (8.1,-1.85) {$\circ$};
\node at (8.05,-2.1) {$z$};


\draw [->] (8.15,0.72)-- (8.8,-0.5);
\node at (8.65,0.2) {$\gamma_1$};
    \node at (8.85,-0.55) {$\bullet$};  

\draw [->] (8.8,-0.55)-- (8.17,-1.8);
\node at (8.65,-1.3) {$\gamma_2$};

\node at (7.5,-2.5) {$2=q\leq p$ \quad or \quad $3=q\leq p \leq 5$};

\node at (7.5,-3) {$I_B:=\langle \gamma_2\gamma_1+\beta_q\cdots\beta_1+\alpha_p\cdots\alpha_1 \rangle$};


\node at (-1.5,-4.25) {$C(p)$};

\node at (1.28,-4.25) {$\bullet$};
\draw [->] (1.2,-4.25)-- (0.5,-4.7);
\node at (0.45,-4.72) {$\bullet$};
\draw [->] (0.45,-4.7)-- (0.45,-5.47);
\node at (0.75,-5.05) {$\rho_p$};
\node at (0.45,-5.55) {$\bullet$};
\draw [->] (0.45,-5.55)-- (0.45,-6.25);
\node at (0.45,-6.3) {$\bullet$};
\draw [->] (0.5,-6.35)-- (1.2,-6.8);
\node at (0.75,-5.85) {$\rho_1$};

\draw [<-] (1.35,-4.25)-- (2,-4.7);
    \node at (2,-4.72) {$\bullet$};  

\draw [dashed,<-] (2.05,-4.8)-- (2.05,-6.2);

\node at (2,-6.3) {$\bullet$};
\draw [->] (1.3,-6.8)--(1.98,-6.35);
\node at (1.25,-6.85) {$\bullet$};

\node at (-0.5,-4.5) {$\circ$};
\node at (-0.5,-4.2) {$a$};
\draw [->] (-0.5,-4.55)-- (0.37,-5.55);
\node at (-0.25,-5.05) {$\alpha$};
\draw [->] (0.37,-5.58)-- (-0.5,-6.45);
\node at (-0.25,-6) {$\beta$};
\node at (-0.5,-6.5) {$\circ$};
\node at (-0.5,-6.7) {$z$};

\node at (1,-7.25) {$\forall p \in \mathbb{Z}_{>0}$};
\node at (1.25,-7.75) {$I_C:=\langle \rho_1\rho_p\rangle$};



\node at (5.75,-4.25) {$D(p,q)$};

\node at (9.08,-4.25) {$\bullet$};
\draw [->] (9,-4.25)-- (8.3,-4.7);
\node at (8.25,-4.72) {$\bullet$};
\draw [->] (8.25,-4.7)-- (8.25,-5.47);
\node at (8.5,-5.05) {$\rho_p$};
\node at (8.25,-5.55) {$\bullet$};
\draw [->] (8.25,-5.55)-- (8.25,-6.25);
\node at (8.25,-6.3) {$\bullet$};
\draw [->] (8.3,-6.35)-- (9,-6.8);
\node at (8.5,-5.85) {$\rho_1$};

\draw [<-] (9.15,-4.25)-- (9.8,-4.7);
    \node at (9.8,-4.72) {$\bullet$};  

\draw [dashed,<-] (9.85,-4.8)-- (9.85,-6.2);

\node at (9.8,-6.3) {$\bullet$};
\draw [->] (9.1,-6.8)--(9.78,-6.35);
\node at (9.05,-6.85) {$\bullet$};

\node at (7.25,-4.5) {$\circ$};
\node at (7.25,-4.2) {$a$};
\draw [->] (7.25,-4.55)-- (8.17,-5.55);
\node at (7.5,-5.05) {$\alpha$};
\draw [->] (8.17,-5.58)-- (7.25,-6.45);
\node at (7.5,-6) {$\beta$};
\node at (7.25,-6.5) {$\circ$};
\node at (7.25,-6.7) {$z$};

\draw [->] (7.25,-4.55)-- (6.5,-5);
\node at (6.75,-4.6) {$\gamma_1$};
\node at (6.5,-5.1) {$\bullet$};
\draw [dashed,->] (6.5,-5.05)-- (6.5,-6);
\node at (6.5,-6.05) {$\bullet$};
\node at (6.7,-6.42) {$\gamma_{q+1}$};
\draw [->] (6.5,-6.05)-- (7.2,-6.45);
\node at (8,-7.25) {$\forall p, \forall q \in \mathbb{Z}_{>0}$};
\node at (7.75,-7.75) {$I_D:=\langle \rho_1\rho_p, \gamma_{q+1} \cdots \gamma_1-\beta \alpha\rangle$};

\end{tikzpicture}
\end{center}
\end{theorem}

From the preceding theorem and our similar results in \cite{Mo}, in Section \ref{Section:non-trivial sufficient conditions for tau-inf} we obtain concrete sufficient criteria that can be efficiently used to verify $\tau$-tilting infiniteness of a large family of algebras. The following pair of theorems give explicit characterizations of $\tau$-tilting infinite algebras in the families $\Mri(\mathfrak{F}_{\sB})$ and $\Mri(\mathfrak{F}_{\nD})$.

\begin{theorem}\cite{Mo}\label{Section:Introduction: tau-inf min-rep-inf bis}
Let $\Lambda=kQ/I$ be a minimal representation-infinite biserial algebra. Then, $\Lambda$ is $\tau$-tilting infinite if and only if it is gentle. That being the case, either $\Lambda=k\widetilde{\mathbb{A}}_m$ (for $m \in \mathbb{Z}_{>0}$) or $I$ is generated by exactly two quadratic relations.
\end{theorem}

Analogously, for the algebras treated here we obtain the following concrete characterization. For further details, see Section \ref{Section:non-trivial sufficient conditions for tau-inf}.

\begin{theorem}\label{Section:Introduction: tau-inf min-rep-inf non-dis}
Let $\Lambda=kQ/I$ be a minimal representation-infinite algebra and non-distributive.
Then $\Lambda$ is $\tau$-tilting infinite if and only if $Q$ has a source or a sink.
\end{theorem}

The above theorem also produces explicit families of minimal $\tau$-tilting infinite algebras.
As the supplement to the last two theorems, we can also give a unified conceptual description of the $\tau$-tilting finite algebras in the families $\Mri(\mathfrak{F}_{\sB})$ and $\Mri(\mathfrak{F}_{\nD})$. Henceforth, for an algebra $\Lambda=kQ/I$, by $R$ we denote a minimal set of uniform relations that generates $I$.
That is every element of $R$ is a linear combination of the form $\sum_{i=1}^{d} \lambda_i p_i$, where $d\in \mathbb{Z}_{>0}$ and each $\lambda_i \in k\setminus \{0\}$, and moreover all paths $p_i$ in $Q$ are of length not less than two such that all $p_i$ start at the same vertex and also end at the same vertex.

\begin{theorem}\label{tau-fin min biserial and non-dist}
Let $\Lambda=kQ/I$ be a minimal representation-infinite algebra which is biserial or non-distributive. Then $\Lambda$ is $\tau$-tilting finite if and only if $\node(\Lambda)\neq \emptyset$ or $|R|=3$.
\end{theorem}

Based on our methodology, we can in fact prove a version of Theorems \ref{Section:Introduction: tau-inf min-rep-inf bis} and \ref{Section:Introduction: tau-inf min-rep-inf non-dis} that allows us to relate our results to the more geometric aspects of representation theory of min-rep-infinite algebras. In particular, as a consequence of Theorems \ref{my conj verified for min-rep non-dist} and \ref{my conj verified for min-rep biserial}, we show the following result.

\begin{theorem}
Let $\Lambda$ be a minimal-representation-infinite algebra which is biserial or non-distributive. Then, $\Lambda$ is $\tau$-tilting infinite if and only if there exists an infinite family of non-isomorphic bricks in $\modu \Lambda$ which are of the same length.
\end{theorem}

In \cite{CKW}, where $k$ is additionally assumed to be of characteristic zero, Chindris-Kinser-Weyman study the module varieties of finite dimensional algebras and introduce some new geometric notions. In particular, the authors say $\Lambda=kQ/I$ has the \emph{dense orbit property} (or $\Lambda$ is DO, for short) if for each dimension vector $\underline{d} \in \mathbb{Z}_{\geq 0}^{|Q_0|}$ and every irreducible component $\mathcal{Z}$ of $\modu (\Lambda,\underline{d})$, $\mathcal{Z}$ has a dense orbit under the action of the general linear group $\GL(\underline{d})$. Moreover, $\Lambda$ is called \emph{Schur-representation-finite} provided that for each dimension vector $\underline{d}$, there are only finitely many orbits of Schur representations in $\modu (\Lambda,\underline{d})$. 

We emphasis that one should not confuse the notion of Schur-representation-finite algebras with those algebras which admit only finitely many isomorphism classes of bricks. However, we show that for the algebras treated in this paper, these notions are the same and we further conjecture this is true in general.

By the brick-$\tau$-rigid correspondence in \cite{DIJ}, if $\Lambda$ is $\tau$-tilting finite, $\modu \Lambda$ contains only finitely many isomorphism classes of bricks, thus $\Lambda$ is Schur-rep-finite. 
For the family of min-rep-infinite algebras studied here, we have the following result.

\begin{theorem}\label{Introduction: proven conjectures}
Suppose $\Lambda$ is a minimal representation-infinite algebra which is biserial or non-distributive. Then, the following hold:

\begin{enumerate}
    \item If $\Lambda$ is DO, then it is $\tau$-tilting finite.
    
    \item $\Lambda$ is Schur-representation-finite if and only if it is $\tau$-tilting finite.
\end{enumerate}
\end{theorem}

Note that the above theorem allows us to view the dense orbit and Schur-rep-finite property through the lens of approximation theory and functorially finite torsion (-free) classes.
Moreover, combining our results with those given in \cite{CKW}, we obtain explicit examples of non-distributive algebras which are $\tau$-tilting finite but not DO. In particular, we show that the converse of the first part in the preceding theorem does not hold, neither for the family $\Mri(\mathfrak{F}_{\sB})$ nor $\Mri(\mathfrak{F}_{\nD})$.

Based upon the above theorem and some further observations given in the last section, we conjecture that the preceding theorem holds in general. See Conjecture \ref{my conjectures}.
From a sequence of implications shown in \cite{CKW} and \cite{DIJ}, we can present our conjectures in the following diagram, where the dotted arrows indicate the statements of Theorem \ref{Introduction: proven conjectures}, which we conjecture that hold in general.

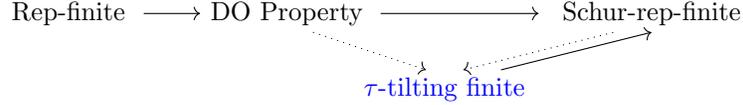
\begin{center}
\begin{tikzpicture}

\node at (-7.25,0) {Rep-finite};
\draw  [->] (-6.25,0) --(-5.5,0);
\node at (-4.35,0) {DO Property};
\draw  [->] (-3.1,0) --(-1,0);
\node at (0.5,0) {Schur-rep-finite};
----
\draw  [dotted,->] (-4,-0.25) --(-2.5,-0.75);
\node at (-2.25,-1) {{\color{blue} $\tau$-tilting finite}};
\draw  [dotted,<-] (-2,-0.75) --(0,-0.25);
----
\draw  [->] (-1.5,-0.75) --(0.5,-0.25);

\end{tikzpicture}
\end{center}

We believe that any result on the aforementioned conjecture will shed new light on the connections between the homological and geometric aspects of representation theory and will provide new tools to attack several open problems posed in \cite{CKW}.

Before we finish this section, let us remark that the second part of the preceding theorem (and Conjecture \ref{my conjectures}) is of the same flavor as the celebrated Brauer-Thrall conjectures. In particular, it implies that for an algebra $\Lambda$, if $\modu \Lambda$ contains infinitely many isomorphism classes of bricks, then for a positive integer $d$ there are infinitely many non-isomorphic bricks of length $d$ over $\Lambda$. This observation also highlights the need for a classification of minimal $\tau$-tilting infinite algebras. This is because one can reduce similar problems about the behavior of bricks to the family of minimal $\tau$-tilting infinite algebras, in the same way the Brauer-Thrall conjectures were reduced to minimal representation-infinite algebras.
\section{Preliminaries and Background}\label{Priliminary}

The materials appearing in this section are often well-known and thus the proofs are omitted. For the less obvious statements we provide some references.

\subsection{Non-distributive algebras}

As recalled earlier, an algebra $\Lambda$ is said to be \emph{distributive} if the lattice of two-sided ideals in $\Lambda$ is distributive. In \cite{Ja}, Jans shows that over an infinite field $k$, the algebra $\Lambda$ is distributive if and only if this lattice is finite.
The following statements are handy and provide a nice characterization of distributive algebras. Although they often hold in a more general setting, unless specified otherwise, throughout this note we always assume $k$ is algebraically closed.
Note that for an algebra $\Lambda=kQ/I$ and a pair of vertices $x$ and $y$ in $Q_0$, the vector space $e_y\Lambda e_x$ always has a natural $e_y\Lambda e_y-e_x\Lambda e_x$ bimodule structure.

\begin{proposition}\label{distributive algebras}
Let $\Lambda=kQ/I$ be an algebra and $x$ and $y$ be two vertices in $Q_0$.
\begin{enumerate}

    \item[\cite{Ja}] The lattice of subbimodules of $e_y\Lambda e_x$ is distributive if and only if $e_y\Lambda e_x$ is a uniserial bimodule.
    
    \item[\cite{Ku}] $\Lambda$ is distributive if and only if the lattice of subbimodules of $e_y\Lambda e_x$ is distributive for every pair of vertices $x$ and $y$ in $Q_0$. In this case, every ideal of $\Lambda$ is principal.
\end{enumerate}
\end{proposition}

\begin{remark}\label{localness of trancations of distributive algebras}

The second statement of the preceding proposition is equivalent to the following characterization: for each $x \in Q_0$, there exists $m(x) \in \mathbb{Z}_{>0}$ such that $e_x\Lambda e_x \simeq k[t]/\langle t^{m(x)}\rangle$ and for any pair of vertices $x$ and $y$ in $Q_0$, we have that $e_y\Lambda e_x$ is cyclic as a left $e_y\Lambda e_y$-module or cyclic as a right  $e_x\Lambda e_x$-module (or both).

Furthermore, from the above properties of distributive algebras, it follows that for any surjective algebra morphism $\phi: \Lambda \rightarrow \Lambda'$, if $\Lambda$ is distributive, then so is $\Lambda'$. Let $\Mri(\mathfrak{F}_{\nD})$ denote the set of those min-rep-infinite algebras which are non-distributive. Observe that every $\Lambda$ in $\Mri(\mathfrak{F}_{\nD})$ is also minimal with respect to the property of being non-distributive, because every rep-finite algebra is distributive.
\end{remark}

For further details on the distributive and non-distributive algebras, we refer the reader to \cite{Ja} and \cite{Ku}.

\subsection{Bongartz's classification}\label{subsection: Bongartz's classification}

In \cite{Bo2}, through a careful study of non-distributive algebras, Bongartz gives a complete classification of non-distributive min-rep-infinite algebras in terms of quivers and relations.
In order to state his classification, let us first fix some notations and recall the notions of resolving a node, as well as gluing, in bound quivers.

Let $\Lambda$ be an algebra with the bound quiver $(Q,I)$. For a vertex $x \in Q_0$, let
$$x^{-}(Q_1):=\{\alpha \in Q_1\,|\, e(\alpha)=x\}$$
be the set of all arrows ending in $x$. Dually, by $x^{+}(Q_1)$ we denote the set of all arrows in $Q$ which start at $x$. Obviously $x^{-}(Q_1) \cap x^{+}(Q_1) \neq \emptyset$ if and only if there is loop at $x$ (i.e, a cyclic path of length one that starts and ends at $x$). 
Moreover, $x \in Q_0$ is a \emph{source} if $x^{-}(Q_1)=\emptyset$ and it is a \emph{sink} if $x^{+}(Q_1)=\emptyset$.
A vertex $x$ is called a \emph{node} provided that it is neither a sink nor a source and if $\alpha \in x^{-}(Q_1)$ and $\beta \in x^{+}(Q_1)$, then $\beta \alpha \in I$. For $\Lambda=kQ/I$, let $\node(\Lambda)$ denote the set of all nodes in $(Q,I)$.

Suppose that $x \in \node (\Lambda)$ and $R$ is a minimal set of generators for $I$ in the bound quiver $(Q,I)$. By \emph{resolving} $x$ we obtain the algebra $\Lambda'=kQ'/I'$, where $(Q',I')$ is given as follows: replace $x$ with two vertices $x^+$ and $x^-$ such that every $\alpha \in x^{-}(Q_1)$ ends in $x^-$ and each $\beta \in x^{+}(Q_1)$ starts at $x^+$. The rest of $Q$ remains untouched and by $Q'$ we denote the resulting quiver. Moreover, let $I'$ be the admissible ideal in $kQ'$, generated by the subset $R'$ of $R$, where $R'$ consists of all relations in $R$ except for those quadratic relations of the form $\beta \alpha$, for all pairs of $\alpha \in x^{-}(Q_1)$ and $\beta \in x^{+}(Q_1)$.

If $(Q',I')$ is obtained from $(Q,I)$ via a sequence of node resolutions at $x_i \in \node(\Lambda)$, it is known that the indecomposable modules in $\modu \Lambda'$ and $\modu \Lambda$ are in one-to-one correspondence, except for the simple modules associated to vertices $x_i$, $x_i^{+}$and $x_i^{-}$. In particular, $\Lambda$ is rep-finite if and only if $\Lambda'$ is so.
As the dual notion, \emph{gluing} a source $a$ with a sink $z$ in a bound quiver $(Q,I)$ is the reverse process of resolving a node: identify $a$ and $z$ to get a new vertex, say $x$, then add all the quadratic relations $\beta\alpha$ for $\alpha \in z^{-}(Q_1)$ and $\beta \in a^{+}(Q_1)$. Obviously, $x$ is a node in the resulting quiver. 
For more details on the node resolution and gluing, see \cite[Section 2.3]{Bo2}. 

For the classification of min-rep-infinite algebras in terms of quivers and relations, Bongartz \cite{Bo2} and Ringel \cite{Ri} substantially benefited from the idea of gluing and resolving nodes, as they both preserve the representation-type of the algebras. However, as we observed in \cite{Mo}, gluing and resolving nodes may change the $\tau$-tilting-type of the algebras (i.e, a $\tau$-tilting infintie algebra may become $\tau$-tilting finite after gluing a sink and a source).

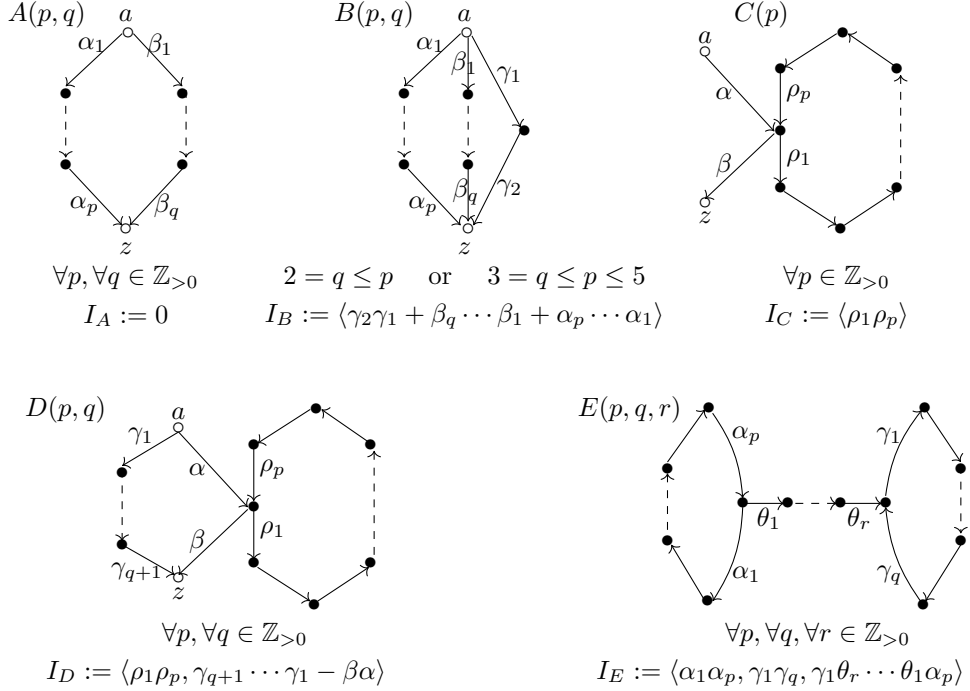
\begin{figure}
\begin{center}
\begin{tikzpicture}

\node at (-1.5,1) {$A(p,q)$};

\node at (-0.42,0.75) {$\circ$};
\node at (-0.45,1) {$a$};
\draw [->] (-0.5,0.75)-- (-1.25,0.05);
\node at (-0.9,0.6) {$\alpha_1$};
\node at (-1.25,-0.05) {$\bullet$};

\draw [dashed,->] (-1.25,0)-- (-1.25,-0.9);

\node at (-1.25,-1) {$\bullet$};
\draw [->] (-1.2,-1.05)-- (-0.5,-1.8);
\node at (-1,-1.55) {$\alpha_p$};

\draw [->] (-0.35,0.75)-- (0.3,0.05);
\node at (0,0.6) {$\beta_1$};
    \node at (0.3,-0.05) {$\bullet$};  

\draw [dashed,->] (0.35,0)-- (0.35,-0.9);

\node at (0.3,-1) {$\bullet$};
\draw [->] (0.3,-1.05)--(-0.4,-1.8);
\node at (0.1,-1.55) {$\beta_q$};

\node at (-0.45,-1.85) {$\circ$};
\node at (-0.45,-2.1) {$z$};
\node at (-0.45,-2.5) {$\forall p, \forall q \in \mathbb{Z}_{>0}$};
\node at (-0.45,-3) {$I_A:=0$};


\node at (2.85,1) {$B(p,q)$};

\node at (4.08,0.75) {$\circ$};
\node at (4.05,1) {$a$};
\draw [->] (4,0.75)-- (3.25,0.05);
\node at (3.6,0.6) {$\alpha_1$};
\node at (3.25,-0.05) {$\bullet$};

\draw [dashed,->] (3.25,0)-- (3.25,-0.9);

\node at (3.25,-1) {$\bullet$};
\draw [->] (3.3,-1.05)-- (4,-1.8);
\node at (3.5,-1.55) {$\alpha_p$};

\draw [->] (4.1,0.7)-- (4.1,0);
\node at (4.05,0.35) {$\beta_1$};
    \node at (4.1,-0.07) {$\bullet$};  

\draw [dashed,->] (4.1,0)-- (4.1,-0.9);

\node at (4.1,-1) {$\bullet$};
\draw [->] (4.1,-1.05)--(4.1,-1.75);
\node at (4.07,-1.35) {$\beta_q$};

\node at (4.1,-1.85) {$\circ$};
\node at (4.05,-2.1) {$z$};


\draw [->] (4.15,0.72)-- (4.8,-0.5);
\node at (4.65,0.2) {$\gamma_1$};
    \node at (4.85,-0.55) {$\bullet$};  

\draw [->] (4.8,-0.55)-- (4.17,-1.8);
\node at (4.65,-1.3) {$\gamma_2$};

\node at (4.05,-2.5) {$2=q\leq p$ \quad or \quad $3=q\leq p \leq 5$};

\node at (4.05,-3) {$I_B:=\langle \gamma_2\gamma_1+\beta_q\cdots\beta_1+\alpha_p\cdots\alpha_1 \rangle$};


\node at (8,1) {$C(p)$};

\node at (9.08,0.75) {$\bullet$};
\draw [->] (9,0.75)-- (8.3,0.3);
\node at (8.25,0.28) {$\bullet$};

\draw [->] (8.25,0.3)-- (8.25,-0.47);
\node at (8.5,-0.05) {$\rho_p$};
\node at (8.25,-0.55) {$\bullet$};
\draw [->] (8.25,-0.55)-- (8.25,-1.25);
\node at (8.25,-1.3) {$\bullet$};
\draw [->] (8.3,-1.35)-- (9,-1.8);
\node at (8.5,-0.9) {$\rho_1$};

\draw [<-] (9.15,0.75)-- (9.8,0.3);
    \node at (9.8,0.28) {$\bullet$};  

\draw [dashed,<-] (9.85,0.2)-- (9.85,-1.2);

\node at (9.8,-1.3) {$\bullet$};
\draw [->] (9.1,-1.8)--(9.78,-1.35);

\node at (9.05,-1.85) {$\bullet$};

\node at (7.25,0.5) {$\circ$};
\node at (7.25,0.7) {$a$};
\draw [->] (7.25,0.45)-- (8.17,-0.55);
\node at (7.5,-0.05) {$\alpha$};
\draw [->] (8.17,-0.58)-- (7.25,-1.45);
\node at (7.5,-1) {$\beta$};
\node at (7.25,-1.5) {$\circ$};
\node at (7.25,-1.7) {$z$};

\node at (9,-2.5) {$\forall p \in \mathbb{Z}_{>0}$};
\node at (9,-3) {$I_C:=\langle \rho_1 \rho_p \rangle$};


\node at (-1.25,-4.25) {$D(p,q)$};

\node at (2.08,-4.25) {$\bullet$};
\draw [->] (2,-4.25)-- (1.3,-4.7);
\node at (1.25,-4.72) {$\bullet$};
\draw [->] (1.25,-4.7)-- (1.25,-5.47);
\node at (1.5,-5.05) {$\rho_p$};
\node at (1.25,-5.55) {$\bullet$};
\draw [->] (1.25,-5.55)-- (1.25,-6.25);
\node at (1.25,-6.3) {$\bullet$};
\draw [->] (1.3,-6.35)-- (2,-6.8);
\node at (1.5,-5.85) {$\rho_1$};

\draw [<-] (2.15,-4.25)-- (2.8,-4.7);
    \node at (2.8,-4.72) {$\bullet$};  

\draw [dashed,<-] (2.85,-4.8)-- (2.85,-6.2);

\node at (2.8,-6.3) {$\bullet$};
\draw [->] (2.1,-6.8)--(2.78,-6.35);
\node at (2.05,-6.85) {$\bullet$};

\node at (0.25,-4.5) {$\circ$};
\node at (0.25,-4.3) {$a$};
\draw [->] (0.25,-4.55)-- (1.17,-5.55);
\node at (0.5,-5.05) {$\alpha$};
\draw [->] (1.17,-5.58)-- (0.25,-6.45);
\node at (0.5,-6) {$\beta$};
\node at (0.25,-6.5) {$\circ$};
\node at (0.25,-6.7) {$z$};

\draw [->] (0.25,-4.55)-- (-0.5,-5);
\node at (-0.25,-4.6) {$\gamma_1$};
\node at (-0.5,-5.1) {$\bullet$};
\draw [dashed,->] (-0.5,-5.05)-- (-0.5,-6);
\node at (-0.5,-6.05) {$\bullet$};
\node at (-0.3,-6.42) {$\gamma_{q+1}$};
\draw [->] (-0.5,-6.05)-- (0.2,-6.45);
\node at (1,-7.25) {$\forall p, \forall q \in \mathbb{Z}_{>0}$};
\node at (0.75,-7.75) {$I_D:=\langle \rho_1\rho_p, \gamma_{q+1} \cdots \gamma_1-\beta \alpha\rangle$};


\node at (6.25,-4.25) {$E(p,q,r)$};

\node at (7.75,-5.5) {$\bullet$}; 
\draw [->] (7.75,-5.55) to [bend left=15](7.35,-6.8);
\node at (7.28,-6.8) {$\bullet$};
\draw [<-] (6.8,-6.05)-- (7.25,-6.8);
\node at (6.75,-6) {$\bullet$};
\draw [dashed,<-] (6.75,-5.1)-- (6.75,-5.9);
\node at (6.75,-5.05) {$\bullet$};
\draw [<-] (7.25,-4.25)-- (6.75,-4.95);
\node at (7.3,-4.23) {$\bullet$};
\draw [->] (7.35,-4.3) to [bend left=15](7.75,-5.43);

\node at (7.8,-6.45) {$\alpha_1$};
\node at (7.8,-4.6) {$\alpha_p$};

\draw [->] (7.8,-5.5)-- (8.3,-5.5);
\node at (8.1,-5.7) {$\theta_1$};
\node at (8.35,-5.5) {$\bullet$};
\draw [dashed,->] (8.45,-5.5)-- (9,-5.5);
\node at (9.05,-5.5) {$\bullet$};
\draw [->] (9.1,-5.5)-- (9.6,-5.5);
\node at (9.3,-5.7) {$\theta_r$};
\node at (9.65,-5.5) {$\bullet$};


\node at (10.17,-4.23) {$\bullet$};
\draw [->]  (9.65,-5.4) to [bend left=15] (10.13,-4.26);
\draw [->] (10.18,-4.25)-- (10.65,-4.95);
\node at (10.65,-5.05) {$\bullet$};  
\draw [dashed,->] (10.65,-5)-- (10.65,-5.9);
\node at (10.65,-6) {$\bullet$};
\draw [->] (10.65,-6.05)--(10.18,-6.8);
\node at (10.15,-6.85) {$\bullet$};
\draw [->]  (10.15,-6.85) to [bend left=15] (9.65,-5.55);

\node at (9.7,-4.6) {$\gamma_1$};
\node at (9.7,-6.45) {$\gamma_q$};


\node at (8.75,-7.25) {$\forall p, \forall q, \forall r \in \mathbb{Z}_{>0}$};

\node at (8.25,-7.75) {$I_E:=\langle \alpha_1\alpha_p, \gamma_1\gamma_q, \gamma_1 \theta_r \cdots \theta_1 \alpha_p\rangle$};

\end{tikzpicture}
\end{center}
    \caption{Bound quivers of the minimal representation-infinite non-distributive algebras.}
    \label{fig:non-distributive quivers}
\end{figure}

Now we are ready to recall the full classification of algebras in $\Mri(\mathfrak{F}_{\nD})$.

\begin{theorem}\cite[Theorem 1]{Bo2}
{\label{Thm:Bongartz classification of min-rep-inf non-distributive}}
An algebra $\Lambda=kQ/I$ is minimal representation-infinite and non-distributive if and only if $(Q,I)$ is one of the algebras of Figure \ref{fig:non-distributive quivers} or those obtained by gluing them.
\end{theorem}

\section{Minimal $\tau$-tilting infinite algebras}\label{Section:Minimal tau-tilting infinite algebras}

As explained in \cite{Mo}, the notion of minimal $\tau$-tilting infinite algebras and our interests in their properties have been significantly motivated by the classical notion of minimal representation-infinite algebra and their prominent role in the representation theory.
It is evident that every min-rep-infinite algebra which is $\tau$-tilting infinite is in fact min-$\tau$-infinite. Since our ultimate goal behind the study of min-rep-infinite algebras is to approach a conceptual classification of min-$\tau$-infinite algebras, we begin this section with a general result on this new concept. In particular, we view min-$\tau$-infinite algebras as a modern counterpart of min-rep-infinite algebras and show a property that we first noticed in the family of min-rep-infinite algebras. 
We remark that this new notion was independently introduced in \cite{Mo} and \cite{Wa}, where the authors treated the $\tau$-tilting finiteness for two separate families of algebras.

Now we prove Theorem 1.1, which plays an important role in the rest of the paper. For the convenience of the reader, we state the theorem in the following.

\begin{customthm}{1.1}\label{no node prop}
If $k$ is an arbitrary field and $\Lambda=kQ/I$ is minimal $\tau$-tilting infinite,
then $\Lambda$ is node-free.
\end{customthm}
\begin{proof}
Suppose $\Lambda=kQ/I$ is min-$\tau$-infinite and let $x\in \node(\Lambda)$. Let $M$ be an indecomposable $\Lambda$-module. If there exists $\alpha \in x^{-}(Q_1)$ and $\beta \in x^{+}(Q_1)$, such that $M(\alpha)\neq 0$ and $M(\beta)\neq 0$, then $S_x$ belongs both to the top and the socle of $M$. Therefore, we can find a nonzero morphism $f \in \End_{\Lambda}(M)$ which is nilpotent. In particular, $M$ is not a brick.

By the ``brick-$\tau$-rigid correspondence" of Demonent-Iyama-Jasso \cite{DIJ}, there are infinitely many bricks in $\modu \Lambda$. Hence, for some $\gamma \in x^{-}(Q_1) \cup x^{+}(Q_1)$, there is an infinite family of bricks in $\modu (\Lambda)$, say $\mathcal{X}$, such that $X(\gamma)=0$, for each $X \in \mathcal{X}$. 
If $J$ is the ideal in $\Lambda$ generated by the arrow $\gamma$, then every $X$ in $\mathcal{X}$ is a brick in $\modu \Lambda/J$. This contradicts the assumption that $\Lambda$ is minimal $\tau$-tilting infinite.

We note that the proof does not need the assumption that $k$ is algebraically closed.
\end{proof}

Before we state a consequence of the previous theorem, let us remark that a similar proof could be used to show that an analogous statement holds in a more general setting. In particular, one can replace the bound quivers with the species subject to relations, in the sense of Gabriel \cite{Ga} and Dlab and Ringel \cite{DR}. Then, the notion of ``node" and its resolution can be defined verbatim and 
analogous results could be shown. Since we will not work in that generality, we leave the details to the reader. 

From the preceding proposition, the next corollary is immediate.

\begin{corollary}\label{gluing min-rep-inf is tau-fin}
Let $\Lambda$ be a minimal representation-infinite algebra. If $\Lambda'$ is obtained from $\Lambda$ via a sequence of gluings, then $\Lambda'$ is $\tau$-tilting finite.
\end{corollary}
\begin{proof}
Since gluing an algebra via a sink and source preserves the representation type, $\Lambda'$ is also min-rep-infinite. If $\Lambda'$ is $\tau$-tilting infinite, then it must be min-$\tau$-infinite, which is impossible by Theorem \ref{no node prop}.
\end{proof}

Observe that the previous result could be used to carry out some investigations on the density of $\tau$-tilting finite algebras among the family of all min-rep-infinite algebras with the same number of simple modules. We do not pursue this direction of research in this note.

In the next proposition we slightly specialize a clever argument made by Bongartz in \cite[Section 1]{Bo1}, such that for certain algebras we can generate infinitely many non-isomorphic bricks of the same length. 
First we briefly recall the main idea.

Suppose $e,f$ are primitive idempotents in $\Lambda$ and let $\mathfrak{R}$ denote the radical of $f\Lambda e$ as $(f\Lambda f-e\Lambda e)$-bimodule. If $(\mathfrak{R}^i)$ is the radical filtration of $f\Lambda e$, either $\dim_k(\mathfrak{R}^i/\mathfrak{R}^{i+1})\leq 1$, for every $i \in \mathbb{Z}_{\geq 0}$, 
or there exists the smallest $l \in \mathbb{Z}_{\geq 0}$ such that for some
$u$ and $v$ in $\mathfrak{R}^l$, the image of $u$ and $v$ in $\mathfrak{R}^l/\mathfrak{R}^{l+1}$ are linearly independent.
From Proposition \ref{distributive algebras}, it follows that for non-distributive algebras there always exists a pair of idempotents $e$ and $f$ such that the latter case holds. Let $N$ be the radical of $\Lambda$ and put $\Lambda':=\Lambda/J$, where $J$ is the ideal generated by $\mathfrak{R}^{l+1}, Nu, uN, Nv, vN$.
Provided the above conditions hold, Bongartz \cite{Bo1} showed that the family of quotient modules $\Lambda' e'/\langle u'-\lambda v' \rangle$ consists of pairwise non-isomorphic indecomposables in $\modu \Lambda$, where $\lambda \in k$, and $e'$, $u'$ and $v'$ are respectively elements of $\Lambda'$ associated to $e$, $u$ and $v$. 
Without loss of generality, one can from the beginning assume $e,f,v$ and $u$ are such that $\Lambda e/\langle u-\lambda v \rangle$ is a one-parameter family of non-isomorphic indecomposable modules.

\begin{proposition}\label{Bongartz's trick for a family of modules}\cite[Section 1]{Bo1}
With the same notation as above, suppose $\Lambda$ is such that for a pair of primitive idempotents $e$ and $f$, we have $\dim_k(\mathfrak{R}^l/\mathfrak{R}^{l+1})>1$. If additionally $e\Lambda e\simeq k$, then  $M_{\lambda}:= \Lambda e /\langle u-\lambda v \rangle$ is a one-parameter family of bricks, where $\lambda \in k$.
\end{proposition}

\begin{proof}
Suppose $e$, $u$ and $v$ are as above. Then, by applying the contraviariant functor $\Hom_{\Lambda}(-,M_\lambda)$ on the canonical surjective morphism $\pi: \Lambda e \twoheadrightarrow M_{\lambda}$ in $\modu \Lambda$, we get the embedding $\End_{\Lambda}(M_{\lambda}) \hookrightarrow \Hom_{\Lambda}(\Lambda e,M_{\lambda})\simeq e M_{\lambda}$. Because $e\Lambda e \simeq k$, we have $e M_{\lambda} \simeq k$ and therefore the assertion follows.
\end{proof}

From the brick-$\tau$-rigid correspondence it follows that if $\Lambda$ satisfies the conditions of the previous theorem, then it is $\tau$-tilting infinite.
Note that the condition $\dim_k(\mathfrak{R}^l/\mathfrak{R}^{l+1})>1$ in the assertion of proposition guarantees that the elements $u$ and $v$ in $\mathfrak{R}^l$ are $\Lambda$-independent, meaning that $u\neq pv$ and $pu\neq v$, for any $p \in \Lambda$.
We should also remark that a similar version of the preceding statement has been employed in \cite{CKW}.

\section{$\tau$-tilting infinite minimal representation-infinite algebras}\label{Section:tau-tilting infinite minimal representation-infinite algebras}

In this section we focus our attention on the family of minimal representation-infinite non-distributive algebras and analyze them from the viewpoint of $\tau$-tilting theory. As a byproduct of showing which algebras in $\Mri(\mathfrak{F}_{\nD})$ are $\tau$-tilting finite and which ones are not, we obtain an infinite family of minimal $\tau$-tilting infinite algebras. This should pave the way for the full classification of min-$\tau$-infinite non-distributive algebras, which we will explore in our future work.

We begin this section by deriving a consequence of 
Theorem \ref{no node prop}.
Thanks to the Bongartz classification in Theorem \ref{Thm:Bongartz classification of min-rep-inf non-distributive}, the following proposition immediately follows from Corollary \ref{gluing min-rep-inf is tau-fin}, thus we omit the proof. 

\begin{proposition}
Let $\Lambda=kQ/I$ be a minimal representation-infinite non-distributive algebra. If $(Q,I)$ has a node, then $\Lambda$ is $\tau$-tilting finite. In particular, every algebra given by gluing of a quiver in Figure \ref{fig:non-distributive quivers} is $\tau$-tilting finite.
\end{proposition}

Let us recall an important result from \cite{Mo} which gives a non-trivial class of $\tau$-tilting finite biserial algebras. In fact, in the aforementioned work it is shown that an infinite family of rep-infinite string algebras, which were christened ``wind wheel" by Ringel \cite{Ri}, are always $\tau$-tilting finite. In particular, the bound quivers $E(p,q,r)$ in Figure \ref{fig:non-distributive quivers} form a very special case of wind wheel algebras. 
We also remark that $A(p,q)$ and $E(p,q,r)$, with the specified orientations and condition given in Figure \ref{fig:non-distributive quivers}, are the only classes of algebras that are common between the families $\Mri(\mathfrak{F}_{\sB})$ and $\Mri(\mathfrak{F}_{\nD})$. Moreover, the characterization of the third family of min-rep-infinite algebras in \cite{Bo2}, which are denoted by $\Mri(\mathfrak{F}_{\gC})$ in the introduction, allows us to divide the family of all min-rep-infinite algebras into disjoint subfamilies. This is because $\big( \Mri(\mathfrak{F}_{\sB}) \cup \Mri(\mathfrak{F}_{\nD}) \big ) \cap \Mri({\mathfrak{F}_{\gC}})= \emptyset$. As mentioned earlier, $\tau$-tilting finiteness of algebras in $\Mri(\mathfrak{F}_{\gC})$ will be treated in our future work.

For further details on wind wheel algebras and a proof of a more general version of the next theorem, see \cite[Section 5]{Mo}.

\begin{proposition}\cite[Proposition 5.1]{Mo}
Suppose $\Lambda$ is an algebra given by the bound quiver $E(p,q,r)$ in Figure \ref{fig:non-distributive quivers}. Then, for every arbitrary choice of $p,q$ and $r$ in $\mathbb{Z}_{>0}$, $\Lambda$ is $\tau$-tilting finite.
\end{proposition}

Observe that from the last two propositions, the list of algebras in $\Mri(\mathfrak{F}_{\nD})$ that we need to analyze with respect to $\tau$-tilting finiteness reduces to those given by the bound quivers $A(p,q)$, $B(p,q)$, $C(p)$ and $D(p,q)$ in Figure \ref{fig:non-distributive quivers}.
The algebras of type $A(p,q)$ are hereditary and it is well-known that they always admit preprojective and preinjective components in their Auslander-Reiten quivers (which are infinite). As remarked in \cite{Mo}, from \cite[VIII.2.7]{ASS} it follows that every rep-infinite algebra which admits a preprojective or preinjective component is $\tau$-tilting infinite. 
In the following proposition we treat the $\tau$-tilting finiteness of the remaining cases.

For the sake of brevity, in the following we say $p$ and $q$ in $\mathbb{Z}_{>0}$ are \emph{admissible} if they satisfy the conditions specified by the bound quivers of Figure \ref{fig:non-distributive quivers}.
\begin{proposition}\label{tau-inf non-dist are schur-rep-inf}
For every admissible choice of $p$ and $q$ in $\mathbb{Z}_{>0}$, the algebras $B(p,q)$, $C(p)$ and $D(p,q)$ are $\tau$-tilting infinite. 
\end{proposition}

\begin{proof}
If $\Lambda=kQ/I$ is any of the algebras $B(p,q)$, $C(p)$ and $D(p,q)$ given by the bound quivers in Figure \ref{fig:non-distributive quivers}, there exist unique source $a$ and unique sink $z$ in the quiver $Q$ such that $e_a \Lambda e_a \simeq k \simeq e_z \Lambda e_z$. 

From the bound quiver of the algebras $B(p,q)$, $C(p)$ and $D(p,q)$, there obviously exist $u$ and $v$ in $\Lambda e_a$ that satisfy the conditions of Proposition \ref{Bongartz's trick for a family of modules}. Hence, for $\lambda \in k$, the family of $\Lambda$-modules given by $M_{\lambda}:= \Lambda e_a /\langle u-\lambda v \rangle$ consists of pairwise non-isomorphic bricks and we are done by brick-$\tau$-rigid correspondence in \cite{DIJ}.
\end{proof}

The following theorem summarizes our results on the $\tau$-tilting finiteness of the family $\Mri(\mathfrak{F}_{\nD})$. In particular, it proves Theorem \ref{Introduction: bound quivers of tau-inf}.

\begin{theorem}\label{tau-inf non-distributive Thm}
Let $\Lambda=kQ/I$ be a minimal representation-infinite algebra which is non-distributive. Then, $\Lambda$ is $\tau$-tilting infinite if and only if the bound quiver $(Q,I)$ is of type $A(p,q)$, $B(p,q)$, $C(p)$ or $D(p,q)$ in Figure \ref{fig:non-distributive quivers}, with an admissible choice of $p$ and $q$ in $\mathbb{Z}_{>0}$.
\end{theorem}

\section{Nontrivial sufficient conditions for $\tau$-tilting infiniteness}\label{Section:non-trivial sufficient conditions for tau-inf}

In this section we use our new results on minimal representation-infinite algebras to give simple characterizations of $\tau$-tilting infinite algebras in the two families $\Mri(\mathfrak{F}_{\sB})$ and $\Mri(\mathfrak{F}_{\nD})$.
As explained in the following, such characterizations give explicit sufficient conditions for $\tau$-tilting infiniteness of a very large family of arbitrary algebras. 

For an algebra $\Lambda$, we will write $\Mri(\Lambda)$ for the set of all quotient algebras of $\Lambda$ which are minimal representation-infinite. 
Obviously,  $\Mri(\Lambda)=\emptyset$ if and only if $\Lambda$ is rep-finite. 
Observe that for each $\Lambda' \in \Mri (\Lambda)$, the category $\modu \Lambda'$ is a full additive subcategory of $\modu \Lambda$.
Thus, if there exists $\Lambda' \in  \Mri(\Lambda)$ which is $\tau$-tilting infinite, then $\Lambda$ is also $\tau$-tilting infinite. Due to the explicit description of algebras in $\Mri(\mathfrak{F}_{\sB})$ and $\Mri(\mathfrak{F}_{\nD})$, the following characterizations will allow us to check the bound quiver of an arbitrary algebra $\Lambda=kQ/I$ and by finding certain subquivers in $(Q,I)$ verify whether it is $\tau$-tilting infinite.

In the following, we use $\Mri(\mathfrak{F}_{\B})$ to denote the family of all min-rep-infintie biserial algebras.
As mentioned earlier, in \cite{Mo} we used the results of Ringel \cite{Ri} and gave a full classification of algebras in $\Mri(\mathfrak{F}_{\B})$ with respect to $\tau$-tilting finiteness. This is because over algebraically closed fields, we have the equality $\Mri(\mathfrak{F}_{\B})=\Mri(\mathfrak{F}_{\sB})$ (see \cite[Section 8.1]{Mo}).
Moreover, we proved that $\Lambda \in \Mri(\mathfrak{F}_{\B})$ is $\tau$-tilting infinite if and only if $\Lambda$ is a gentle algebra \cite[Theorem 1.4]{Mo}.
Moreover, we achieved an explicit family of minimal $\tau$-tilting infinite algebras which are not min-rep-infinite \cite[Theorem 1.2]{Mo}. This proved that our new generalization of this classical notion to ``minimal $\tau$-tilting infinite" is novel and deserves thorough investigations. In fact, a classification of min-$\tau$-infinite special biserial algebra will appear in our forthcoming work.

To begin the comparison between the $\tau$-tilting infinite algebras in $\Mri(\mathfrak{F}_{\sB})$ and those from $\Mri(\mathfrak{F}_{\nD})$, we recall the following theorem from our previous work.

\begin{theorem}\cite[Corollary 6.7]{Mo}
Let $\Lambda=kQ/I$ be a minimal representation-infinite biserial algebra. Then, $\Lambda$ is $\tau$-tilting infinite if and only if $\Lambda=k\widetilde{\mathbb{A}}_m$ (for some $m\in \mathbb{Z}_{>0}$), or $(Q,I)$ is a bound quiver of the following form:

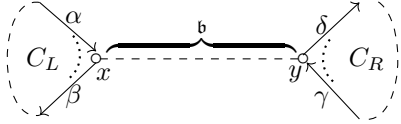
\begin{center}
\begin{tikzpicture}
 \draw [->] (1.25,0.75) --(2,0.1);
    \node at (1.7,0.55) {$\alpha$};
 \draw [<-] (1.25,-0.75) --(2,-0.05);
    \node at (1.7,-0.5) {$\beta$};
  \draw [dashed] (1.25,0.75) to [bend right=100] (1.25,-0.75);
   \node at (1.3,0) {$C_L$};
    \node at (2,0) {$\circ $};
    \node at (2.1,-0.2) {$x$};
    
    \draw [dotted,thick] (1.65,-0.25) to [bend right=50](1.65,0.35);
 \draw [dashed] (2.05,0) --(4.7,0);
 
 \node at (3.4,0.3) {$\overbrace{\qquad \qquad\quad \qquad}^{\mathfrak{b}}$};
 
 \node at (4.75,0.0) {$\circ$};
 \draw [<-] (5.5,0.75) --(4.75,0.05);
    \node at (5,0.45) {$\delta$};
 \draw [->] (5.50,-0.8) --(4.8,-0.05);
    \node at (5,-0.55) {$\gamma$};
  \draw [dashed] (5.55,0.8) to [bend left=100] (5.55,-0.8);
   \node at (5.6,0) {$C_R$};
   \node at (4.65,-0.2) {$y$};
   
\draw [dotted,thick] (5.15,-0.3) to [bend left=50](5.15,0.3);
\end{tikzpicture}
\end{center}
where the bar $\mathfrak{b}$ is a non-uniserial copy of $\mathbb{A}_m$ (for some $m \in \mathbb{Z}_{>0}$), and $ I= \langle \beta \alpha , \delta \gamma \rangle$.
\end{theorem}

\begin{remark}
In the preceding theorem, if $\Lambda=k\widetilde{\mathbb{A}}_m$ (for some $m\in \mathbb{Z}_{>0}$), then obviously $\widetilde{\mathbb{A}}_m$ has an acyclic orientation. 
Moreover, the bound quiver $(Q,I)$ is aptly called \emph{barbell} in Ringel's work \cite{Ri}. 
Observe that for any $\tau$-tilting infinite algebra $\Lambda=kQ/I$ in the above theorem, $Q$ always has at least a sink or a source. 
If $C_L= \alpha \cdots \beta$ and $C_R=\gamma \cdots \delta$ denote the cyclic strings illustrated above, they share no vertex and they can be of any positive length.  In particular, $C_L$ and/or $C_R$ can be loops.
Note that, unlike the bound quivers $E(p,q,r)$ in Figure \ref{fig:non-distributive quivers}, the cyclic strings $C_L$ and $C_R$ need not be uniserial. Namely, those substrings of $C_L$ and $C_R$ that are shown by the dashed segments can have any orientation of the arrows. For a more general version of the above theorem, see \cite{Mo}.
\end{remark}

As mentioned before, for an algebra $\Lambda=kQ/I$, by $R$ we denote a minimal set of uniform relations that generates $I$. 
Although $R$ is not necessarily unique, the number of relations in $R$ is an invariant, which we denote by $|R|$. In fact, for each pair of vertices $x, y \in Q_0$, if $I(x,y)$ denotes the relations in $I$ determined by the linear combinations of paths from $x$ to $y$, then $r(x,y):=|R\cap I(x,y)|$ is an invariant of $\Lambda$, because we have $r(x,y)=\dim_k \Ext^2_{\Lambda}(S_x,S_y)$. Now the following result immediately follows from the previous theorem and our results in \cite{Mo}.

\begin{proposition}
Let $\Lambda=kQ/I$ be a minimal representation-infinite biserial algebra. Then, $\Lambda$ is $\tau$-tilting infinite if and only if $R$ consists of at most two relations.
\end{proposition}

Similarly, we have the following characterization for the $\tau$-tilting infinite algebras in $\Mri(\mathfrak{F}_{\nD})$, given in Theorem \ref{Section:Introduction: tau-inf min-rep-inf non-dis}.

\begin{customthm}{1.4}
Let $\Lambda=kQ/I$ be a minimal representation-infinite non-distributive algebra. Then, $\Lambda$ is $\tau$-tilting infinite if and only if $Q$ has a sink or a source.
\end{customthm}
\begin{proof}
The assertion is a direct consequence of Theorem \ref{Thm:Bongartz classification of min-rep-inf non-distributive} and Theorem \ref{tau-inf non-distributive Thm}.
\end{proof}

From the preceding characterizations of $\tau$-tilting infinite algebras in $\Mri(\mathfrak{F}_{\sB})$ and $\Mri(\mathfrak{F}_{\nD})$, now we show Theorem \ref{tau-fin min biserial and non-dist}, which somehow supplements the above statements. 
Recall that for an algebra $\Lambda=kQ/I$, by $\node(\Lambda)$ we denote the set of nodes in $(Q,I)$.

\begin{customthm}{1.5}
Let $\Lambda=kQ/I$ be a minimal representation-infinite algebra which is biserial or non-distributive. Then $\Lambda$ is $\tau$-tilting finite if and only if $\node(\Lambda)\neq \emptyset$ or $|R|=3$.
\end{customthm}
\begin{proof}
Suppose $\Lambda$ is non-distributive. By Theorem \ref{tau-inf non-distributive Thm}, $\Lambda$ is $\tau$-tilting finite if and only if $(Q,I)$ is obtained via gluing the bound quivers in Figure \ref{fig:non-distributive quivers}, or $(Q,I)$ is the quiver of type $E(p,q,r)$. In the former case we have $\node(\Lambda)\neq \emptyset$, while in the latter case we get $|R|=3$. So, we have the desired result.

If $\Lambda$ is biserial, the assertion follows from the classification of $\tau$-tilting finite min-rep-infinite biserial algebras given in \cite[Table 1, Section 7]{Mo}.
\end{proof}

Let us remark that the condition $|R|=3$ in the preceding theorem only takes care of the class of wind wheel algebras, as the only $\tau$-tilting finite class of algebras in $\Mri(\mathfrak{F}_{\sB})\cup \Mri(\mathfrak{F}_{\nD})$ whose bound quivers is node-free.
In other words, the above characterization can be phrased as following.

\begin{corollary}\label{Cor: explicit criterial for tau-inf min-rep-inf}
Suppose $\Lambda=kQ/I$ is a minimal representation-infinite algebra which is biserial or non-distributive. Then it is $\tau$-tilting finite if and only if $(Q,I)$ has a node or a non-quadratic monomial relation.
\end{corollary}

We finish this section by the following consequence of our results, combined with the classification of Happel and Vossieck in \cite{HV}. 
In the aforementioned work the authors study those rep-infinite algebras $\Lambda$ for which $\Lambda/\langle e\rangle$ is rep-finite, for every primitive idempotent $e \in \Lambda$. We call such algebras \emph{weakly minimal representation-infinite}. Observe that every min-rep-infinite algebra is evidently weakly min-rep-infinite, but the converse is not necessarily true. In \cite{HV}, the authors determined which weakly min-rep-infinite algebras admit a preprojective components. Comparing their list with the algebras treated in this note, we get the following result.

\begin{theorem}
Let $\Lambda=kQ/I$ be a minimal representation-infinite algebra which is biserial or non-distributive. Then, $\Lambda$ admits a preprojective component if and only if $\Lambda$ is hereditary, or $(Q,I)=B(p,q)$ for an admissible choice of $p$ and $q$ in Figure \ref{fig:non-distributive quivers}.
\end{theorem}

In other words, the previous theorem states that among all the $\tau$-tilting infinite algebras in $\Mri(\mathfrak{F}_{\sB}) \cup \Mri(\mathfrak{F}_{\nD})$, only those which are hereditary or canonical will admit a preprojective component. Observe that one can analogously define the notion of ``weakly minimal $\tau$-tilting infinite" algebra in the sense of the definition given above. Then, by results of \cite{Mo}, we easily conclude that the family of all such algebras strictly contains  all the algebras classified by Happel and Vossieck \cite{HV}.

Before we end this section, we wish to draw the attention of the reader to another interesting connection between the two families of algebras $\Mri(\mathfrak{F}_{\sB})$ and $\Mri(\mathfrak{F}_{\nD})$. In particular, using Bongart'z classification and relatively simple techniques from degeneration of algebras, one can show that every algebra in $\Mri(\mathfrak{F}_{\nD})$ degenerates to a special biserial algebra. Hence, $\Mri(\mathfrak{F}_{\sB})\cup \Mri(\mathfrak{F}_{\nD})$ consists of only tame algebras and every wild min-rep-infinite algebra belongs to $\Mri(\mathfrak{F}_{\gC})$.

\section{$\tau$-tilting theory and geometry of module varieties}

As mentioned in the introduction, in \cite{CKW}, the authors studied the representation theory of finite dimensional algebras from the geometric perspectives. In particular, for an algebra $\Lambda=kQ/I$ and a fixed dimension vector $ \underline{d} \in \mathbb{Z}^{Q_0}$, they looked at several interesting properties of the module varieties $\modu(\Lambda, \underline{d})$ and explored the connections between them. In the following we mainly consider two of the notions introduced there and postpone a more extensive study of the subject to our future investigations. 
In particular, we use our results on the $\tau$-tilting theory of min-rep-infinite algebras to propose a long-term project that investigates the potential connections between $\tau$-tilting theory and some geometric aspects of representation theory. Any progress in this direction fortifies the linkages between categorical and geometric aspects of the subject and opens new horizons.

Let us begin this section by recalling some definitions and result from \cite{CKW}. For the sake of brevity, we do not aim to present a self-contained section. Hence, for the most part, we refer to the aforementioned paper and the references there. As before, $k$ is supposed to be an algebraically closed field. Moreover, for the statements that we recall from \cite{CKW}, as it is assumed by the authors, we suppose $k$ is of characteristic zero.

\subsection{Dense orbit property}
As defined earlier, an algebra $\Lambda$ has the \emph{dense orbit property} (or $\Lambda$ is DO, for short) if for every dimension vector $\underline{d}$, and each irreducible component $\mathcal{Z}$ in $\modu (\Lambda, \underline{d})$, there exists a dense orbit in $\mathcal{Z}$. 

An algebra $\Lambda$ is obviously DO if every irreducible component $\mathcal{Z}$ in $\modu (\Lambda, \underline{d})$ contains a rigid $\Lambda$-module. However, the converse is not necessarily true.
It is straight forward to check that every rep-finite algebra is DO. To show that the DO property is novel, in \cite{CKW} the authors construct explicit examples of rep-infinite algebras which are DO. In fact, all of the examples given there are two-point algebras and have infinite global dimension.
Via an explicit checking, which we leave to the reader, one can verify that all the DO algebras given in \cite{CKW} admit only finitely many isomorphism classes of bricks, thus they are $\tau$-tilting finite. This could be alternatively verified by the new results of Wang \cite{Wa} on two-point algebras.

The conjectures we state in the following are primarily motivated by the above observation and our results on the family of min-rep-infinite algebras. In particular, the following result holds for the family of string algebras, thus it closely relates to our work in \cite{Mo}.

\begin{proposition}[{\cite[Proposition 4.4]{CKW}}]\label{DO String}
A string algebra is DO if and only if it is representation-finite.
\end{proposition}

Because all the min-rep-inf biserial algebras are string algebras, from the preceding proposition it immediately follows that if $\Lambda \in \Mri(\mathfrak{F}_{\B})$, then $\Lambda$ is never DO. We show that a similar fact holds for the family $\Mri(\mathfrak{F}_{\nD})$, meaning that if $\Lambda$ is min-rep-infinite non-distributive, then $\Lambda$ is never DO.
To prove this, we first need to recall another geometric notion introduced in \cite{CKW}.

\subsection{Schur-representation-finite}
As defined in the introduction, an algebras $\Lambda$ is \emph{Schur-representation-finite} provided that for each dimension vector $\underline{d}$, there are only finitely many orbits of Schur representations (i.e, bricks) in $\modu(\Lambda, \underline{d})$.

The following statement shows how these two new geometric notions are related. 

\begin{lemma}[{\cite[Lemma 3.2]{CKW}}]\label{DO is Schur-rep-fin}
If $\Lambda$ is DO, then it is Schur-representation-finite.
\end{lemma}

Before we focus on the family of min-rep-infinite non-distributive algebras, let us state a relevant result from \cite{CKW}. 
For the convenience of the reader, we also present a sketch of the proof.
Recall that $\Lambda=kQ/I$ is called \emph{triangular} if $Q$ has no oriented cycles.

\begin{theorem}[{\cite[lemma 4.3]{CKW}}]\label{CKW Thm on triangualr non-dist}
Suppose $\Lambda$ is a non-distributive triangular algebra. Then $\Lambda$ admits infinitely many Schur representations of the same dimension. In particular, $\Lambda$ is not Schur-representation-finite, thus it is not DO.
\end{theorem}

\begin{proof}
Since $\Lambda$ is non-distributive, there exist primitive idempotents $e,f$ in $\Lambda$ and a smallest $l\in \mathbb{Z}_{>0}$ such that
$\dim_k(\mathfrak{R}^l/\mathfrak{R}^{l+1})>1$, where $\mathfrak{R}$ denotes the radical of $f\Lambda e$ as $(f\Lambda f-e\Lambda e)$-bimodule.
Without loss of generality, by Proposition \ref{Bongartz's trick for a family of modules} and the paragraph preceding that,
we can assume $\mathfrak{R}^{l+1}=0$.
Moreover, because $\Lambda$ is triangular, $e \neq f$ and we have $f\Lambda f\simeq k \simeq e\Lambda e$, thus $f \Lambda e$, viewed as a $(f\Lambda f-e\Lambda e)$-bimodule, is in fact a vector space. Thus, its radical is in fact zero. 

Therefore, as in Proposition \ref{Bongartz's trick for a family of modules}, we can find $u$ and $v$ in $\Lambda e$ such that for every $\lambda \in k$, the $\Lambda$-module $M_{\lambda}:= \Lambda e /\langle u-\lambda v \rangle$ is a brick.

\end{proof}

The proof of the preceding theorem is based on the insightful construction of Bongartz \cite[Section 1]{Bo1}, which also allows us to show the following result.

\begin{theorem}\label{my conj verified for min-rep non-dist}
Let $\Lambda$ be a minimal representation-infinite non-distributive algebra. Then $\Lambda$ is $\tau$-tilting finite if and only if it is Schur-representation-finite. In particular, if $\Lambda$ is $\tau$-tilting infinite, then it is not DO.
\end{theorem}

\begin{proof}
First we note that by the $\tau$-rigid-brick correspondence in \cite{DIJ}, every $\tau$-tilting finite algebra is Schur-representation-finite. Hence, to prove the first assertion, we only show that for a min-rep-infinite non-distributive algebra, if $\Lambda$ is Schur-rep-finite, then it is $\tau$-tilting finite.

Suppose $\Lambda=kQ/I$. By Theorem \ref{Thm:Bongartz classification of min-rep-inf non-distributive}, $(Q,I)$ is one of the bound quivers in Figure \ref{fig:non-distributive quivers} or their glued version. Moreover, from Theorem \ref{tau-inf non-distributive Thm}, $\Lambda$ is $\tau$-tilting infinite if and only if $(Q,I)$ is of type $A(p,q)$, $B(p,q)$, $C(p)$ or $D(p,q)$ in Figure \ref{fig:non-distributive quivers}, with an admissible choice of $p$ and $q$ in $\mathbb{Z}_{>0}$.

For any acyclic quiver $\widetilde{\mathbb{A}}_{m}$, there exists a one-parameter family of band modules of the same length, sitting on the mouths of the band tubes. This shows the desired result for type $A(p,q)$. As for the other cases, note that in Proposition \ref{tau-inf non-dist are schur-rep-inf} we also gave an explicit one-parameter family of bricks of the same length for the bound quivers $B(p,q)$, $C(p)$ or $D(p,q)$.
This proves the first assertion. The second assertion immediately follows from Lemma \ref{DO is Schur-rep-fin}.

\end{proof}

Note that, unlike Theorem \ref{CKW Thm on triangualr non-dist}, in the family of min-rep-inf non-distributive algebras, which are concerned in the above theorem, there are families of algebras which are not triangular (see Figure \ref{fig:non-distributive quivers}).

The following result, which treats the family of min-rep-infinite biserial algebras, is analogous to the previous theorem.
Here we use our results in \cite{Mo}, which give explicit classification of $\tau$-tilting infinitene algebras in $\Mri(\mathfrak{F}_{\sB})$ to prove the assertion.

\begin{theorem}\label{my conj verified for min-rep biserial}
Let $\Lambda$ be a minimal representation-infinite biserial algebra. Then $\Lambda$ is $\tau$-tiling finite if and only if it is Schur-representation-finite. In particular, if $\Lambda$ is $\tau$-tilting infinite, then it is not DO.
\end{theorem}

\begin{proof}
We only need to show that if $\Lambda$ is $\tau$-tilting infinite, there exists a dimension vector $\underline{d}$ such that $\modu (\Lambda, \underline{d})$ contains infinitely many bricks.
By \cite[Theorem 1.3]{Mo}, we know that $\Lambda=kQ/I$ is $\tau$-tilting infinite if and only if $\Lambda=k\widetilde{\mathbb{A}}_m$, for some $m \in \mathbb{Z}_{>0}$, or $(Q,I)$ is a barbell algebra. In the former case, the desired result is well-known. As for the latter case, one can easily check that the band given in the proof of \cite[Proposition 5.6]{Mo} admits a one parameter family of band modules of the same length which are brick. The last assertion follows from Lemma \ref{DO is Schur-rep-fin}.
\end{proof}

From the above results, Theorem \ref{Introduction: proven conjectures} follows.

\begin{customthm}{1.7}
Suppose $\Lambda$ is a minimal representation-infinite algebra which is biserial or non-distributive. Then, the following hold:

\begin{enumerate}
    \item If $\Lambda$ is DO, then it is $\tau$-tilting finite.
    
    \item $\Lambda$ is Schur-representation-finite if and only if it is $\tau$-tilting finite.
\end{enumerate}

\end{customthm}

Based upon the above results and some other classes of algebra in $\Mri(\mathfrak{F}_{\gC})$ that we have considered, we state the following conjectures. In the rest of the section we will provide some further observations on these conjectures that outline our future work. 

\begin{conjecture}\label{my conjectures}
For each algebra $\Lambda$, the following hold:
\begin{enumerate}
    \item If $\Lambda$ is DO, then it is $\tau$-tilting finite.
    \item $\Lambda$ is Schur-representation-finite if and only if $\Lambda$ is $\tau$-tilting finite.
\end{enumerate}
\end{conjecture}

From Lemma \ref{DO is Schur-rep-fin}, it is immediate that if the second part of the previous conjecture holds, then the first assertion also follows.
To better explain our motivations for our future investigations in the direction of above conjectures, let us briefly recall another geometric notion studied in \cite{CKW}, which allows us to give a more complete picture. 

An algebra $\Lambda$ is said to have \emph{multiplicity-free property} (or $\Lambda$ is MF, for short), if for every dimension vector $\underline{d}$, and each irreducible component $\mathcal{Z}$ of the module variety $\modu( \Lambda, \underline{d})$, the algebra of semi-invariants $k[\mathcal{Z}]^{\SL(\underline{d})}$ is multiplicity-free.
This holds for $\Lambda=kQ/I$ if and only if for each weight $\theta \in \mathbb{Z}^{Q_0}_{\geq 0}$ we have $\dim_k (\SI(\mathcal{Z})_{\theta}) \leq 1$, where $\SI(\mathcal{Z})_{\theta}$ is the space of semi-invariants on $\mathcal{Z}$ of weight $\theta$. For further details, see \cite{CKW}.

Now, in the following diagram we can present a more complete version of the sequence of implications mentioned in the Introduction. Here the dotted arrows indicate the statements of Conjecture \ref{my conjectures}.
\begin{center}
\begin{tikzpicture}

\node at (-8.1,0.4) {($\dagger$)};

\node at (-7.25,0) {Rep-finite};
\draw  [->] (-6.25,0) --(-5.5,0);
\node at (-4.35,0) {DO Property};
\draw  [->] (-3.1,0) --(-1,0);
\node at (0.5,0) {Schur-rep-finite};
\draw  [->] (1.95,0) --(2.75,0);
\node at (4,0) {MF Property};
----
\draw  [dotted,->] (-4,-0.25) --(-2.5,-0.75);
\node at (-2.25,-1) {{\color{blue} $\tau$-tilting finite}};
\draw  [dotted,<-] (-2,-0.75) --(0,-0.25);
----
\draw  [->] (-1.5,-0.75) --(0.5,-0.25);

\end{tikzpicture}
\end{center}

Note that Theorems \ref{my conj verified for min-rep non-dist} and \ref{my conj verified for min-rep biserial} show that for min-rep-infinite non-distributive and biserial algebras, the implications in ($\dagger$) indicated by dotted arrows actually hold.
Moreover, if $\mathfrak{F}$ is the family of all algebra $\Lambda$ such that $\Mri(\Lambda)$ contains a $\tau$-tilting infinite algebra in $\Mri(\mathfrak{F}_{\sB}) \cup \Mri(\mathfrak{F}_{\nD})$, then every such $\Lambda$ is Schur-rep-infinite, thus both parts of Conjecture \ref{my conjectures} hold for the family $\mathfrak{F}$. Here, as before, by $\Mri(\Lambda)$ we denote the set of all quotient algebras of $\Lambda$ which are min-rep-infinite.
In general, the implications on the top row of ($\dagger$) have been already proved in \cite{CKW}, where the authors also show the following:
\begin{theorem}\cite{CKW} \label{CKW preproj & tame Thms}
Let $\Lambda$ be an algebra.
\begin{enumerate}
    \item If $\Lambda$ admits a preprojective (preinjective) component, all of the properties on the top row of $(\dagger)$ are equivalent;
    
    \item For the class of tame algebras, $\Lambda$ is Schur-rep-finite if and only if it is MF. 
\end{enumerate}
\end{theorem}
Moreover, the authors conjectured that the equivalence in the second part of the previous theorem holds for arbitrary algebras (i.e, an algebra is Schur-rep-finite if and only if it is MF).

\begin{remark}
We finish this section by a list of observations and motivations that may provide further insights into Conjecture \ref{my conjectures}, as well as to those given by Chindris-Kinser-Weyman \cite{CKW}.
\begin{enumerate}
    \item[$(1')$] As mentioned earlier, a rep-infinite algebra $\Lambda$ which admits a preprojective or preinjective component is $\tau$-tilting infinite. Thus, statement $(1)$ in Theorem \ref{CKW preproj & tame Thms} holds for the entire diagram ($\dagger$).
    
    \item[$(2')$] If the second part of Conjecture \ref{my conjectures} holds, statement $(2)$ in Theorem \ref{CKW preproj & tame Thms} will imply that every tame algebra is $\tau$-tilting finite if and only if it is MF. Thus, the class of tame algebras can be the first candidate for verification of the aforementioned part.
    Moreover, our conjecture may provide new tools from $\tau$-tilting theory to verify the conjecture of Chindris-Kinser-Weyman that follows Theorem \ref{CKW preproj & tame Thms}.
    
    \item[$(3')$] As shown in \cite{Mo}, the class of wind wheel algebras, introduced in \cite{Ri}, forms rep-infinite string algebras which are $\tau$-tilting finite. Therefore, by Proposition \ref{DO String}, they serve as explicit examples to show that the conjectural implication ``DO property implies $\tau$-tilting finite", denoted by a dotted arrow in ($\dagger$), cannot be an equivalence.
    
    \item[$(4')$] Representation-finiteness and $\tau$-tilting finiteness of algebras could be described in terms of functorially fintie subcategories in the module category and in the language of approximation theory.
    Hence, showing the second part of Conjecture \ref{my conjectures} gives an immediate categorification of the geometric notion of Schur-rep-finite. Moreover, if any part of Conjecture \ref{my conjectures} holds, we can sandwich the DO property between two homological notions and hope to obtain a categorification of this geometric notion.
\end{enumerate}    
\end{remark}

\vskip 0.7 cm
\noindent
\textbf{Acknowledgements.} The author would like to thank Hugh Thomas, Charles Paquette and Pierre-Guy Plamondon for several helpful conversations and comments. The author is also grateful to the organizers of the CIMPA School on ``Geometric and Homological Methods in the Representation Theory of Associative Algebras and their Applications", held in Medellin, Colombia, as well as to Ryan Kinser for a series of talks on the subject considered in the final section of this note.

\end{document}